\documentclass[12pt]{article}
\usepackage{amsmath,amssymb,epsfig,setspace,color,graphicx,graphics,amsthm,cancel}
\usepackage[left=2.4cm,top=2.4cm,right=2.4cm,bottom=2.5cm]{geometry}
\newtheorem{theorem}{Theorem}[section]
\newtheorem{lemma}[theorem]{Lemma}
\newtheorem{observation}[theorem]{Observation}
\newtheorem{question}[theorem]{Question}
\newtheorem{corollary}[theorem]{Corollary}
\newtheorem{proposition}[theorem]{Proposition}

\newtheorem{defn}[theorem]{Definition}

\DeclareMathOperator{\len}{len}
\DeclareMathOperator{\rd}{rd}

\newcommand{\x}{ \, \square  \,}

\title{The $2$-burning number of a graph}
\author{C.B. Jacobs\thanks{Wellesley College, MA, USA}, M.E. Messinger\thanks{Mount Allison University, NB, Canada}, A.N. Trenk\thanks{Wellesley College, MA, USA, Corresponding author, \tt{atrenk@wellesley.edu}}}

\date{\today}
\begin{document}
\maketitle

\begin{abstract}
  We study a discrete-time model for the spread of information in a graph, motivated by the idea that people believe a story when they learn of it from two different origins.    Similar to the burning number,  in this problem, information spreads in rounds and a new source can appear in each round.  For a graph $G$, we are interested in $b_2(G)$, the   minimum  number of rounds until the information has spread  to all vertices of graph $G$.   We are also interested in finding $t_2(G)$,  the minimum  number of sources necessary  so that the information spreads to all vertices of $G$ in $b_2(G)$ rounds.   In addition to general results, we find $b_2(G)$ and $t_2(G)$ for the classes of spiders and wheels and show that their behavior differs with respect to these two parameters.  We also provide examples and prove upper bounds
  for these parameters for Cartesian products of graphs.
\end{abstract}


\section{Introduction}
\label{intro-sec}

The concept of graph burning, introduced by Bonato et al.~\cite{BJR2014} is a deterministic, discrete-time model for the spread of social contagion on a graph.  A social network can be modeled by a graph in which the vertices represent people and the edges represent relationships.    For example, on a graph whose vertices correspond to Facebook users, edges may represent users who are ``Facebook friends".  Information, such as gossip, rumors, or memes, can spread from vertex to vertex over time along edges of the graph and we model the process using discrete time-steps called \emph{rounds}.  Such information may not stem from a single source vertex; there may be a number of sources that appear over time in the graph.    Authors often refer to vertices as ``unburned" (unaware of the information) and ``burned" (aware of the information) since rumors can appear to spread  swiftly like fire.  We will instead use colors to depict these states:  uncolored for unburned and blue for burned.  
 The discrete-time spread of information also mimics the spread of fire in the Firefighter Problem~\cite{FirefighterSurvey}, although the latter involves agents try to block the spread of fire on a graph.

In 2016, Bonato et al.~\cite{BJR2016} suggested a generalized burning process; and such a  study   was formally initialized in 2021 by Li et al.~\cite{LQL2021}.  For a finite graph $G$, there are two possible states for a vertex: uncolored or blue. 
Initially, at round $0$, all vertices are uncolored.  At round $j$ for $j \ge 1$, one or both of the following  occur:  (i) an uncolored vertex is selected as a source and colored blue, and (ii) every uncolored vertex that had at least $r$ blue neighbors at round $j-1$ is colored blue. 

If $r=1$, the process describes the original burning model of~\cite{BJR2016}, which has been studied extensively; see the survey~\cite{Bsurvey}.  One of the main questions surrounding the original burning model is how quickly the influence or contagion can propagate through the graph.
For a finite graph $G$, Li et al.~\cite{LQL2021} denote by $b_r(G)$, the minimum number of rounds after which every vertex is colored blue and since $G$ is finite, the parameter is well-defined.  Li et al.~\cite{LQL2021} named the parameter, the {\it generalized burning number}, but because this label does not reference $r$ and $b_r(G)$ can change for different values of $r$ on a fixed graph $G$, we refer to   the parameter  as the {\it $r$-burning number} of graph $G$.  For  example, for the graph $H$   in Figure~\ref{fig:firstexample} it is straightforward to show $b_2(H)=3$   and this can be achieved  by selecting $a,b,c$ as sources in rounds $1,2,3$ respectively.  If these same sources are selected in the order $a,c,b$, it requires four rounds for all vertices to turn blue.

\begin{figure}[htbp]
\[ \includegraphics[width=0.125\textwidth]{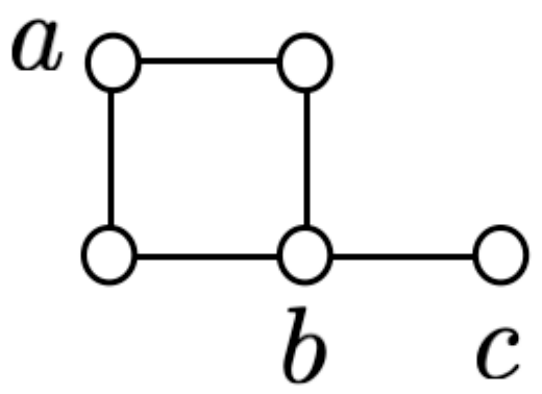}\]

\label{fig:firstexample} 

\caption{A graph $H$ with $b_2(H) = 3$. }
\end{figure}

The generalized burning process is related to  $r$-neighbor bootstrap percolation, which is another way to model the spread of infection or information among vertices in a graph.  In both processes, we color infected vertices blue, but in a generalized burning process there is at most one source vertex per round, whereas in bootstrap percolation, the source vertices are all selected at the start.  More precisely,  for a graph $G$ and positive integer $r$, the process of $r$-neighbor bootstrap percolation is the following: during round $0$, a set of vertices $A \subseteq V(G)$ turns blue; and during round $t>0$, every vertex that has at least $r$ blue neighbors at round $t-1$, becomes blue.  For a set $A$ in graph $G$, {\it percolation} of $G$ occurs if every vertex of $G$ is eventually blue and in this case, $A$ is called a {\it percolating set}.  
 Consequently, for a fixed positive integer $r$,  the cardinality of a minimum  size percolating set on graph $G$ provides a lower bound for the $r$-burning number.  We discuss the relationship of the $r$-burning number to $r$-neighbor bootstrap percolation further in Section~\ref{sec:questions}.  For more background on $r$-neighbor bootstrap percolation, see \cite{BBM,MN,PS}.  
 \medskip

In this paper, we focus on $b_2(G)$, the $2$-burning number of graph $G$, motivated by the idea that people often  believe a rumor when they hear it from two different people.  Li et al.~\cite{LQL2021} determined the $2$-burning number for some families of graphs including paths, cycles,  and complete bipartite graphs.  They also provided some preliminary bounds on $b_2(G)$ for a general graph $G$.   We also study $t_2(G)$, the $2$-burning source number of a graph $G$, which we define as the minimum number of sources so that all vertices of $G$ are blue after $b_2(G)$ rounds.  The parameter $t_2(G)$  provides a middle ground between the $2$-burning number of $G$ and the minimum size percolating set  for $2$-neighbor bootstrap percolation on $G$.

 The rest of the paper is organized as follows.  In Section 2 we provide the foundational definitions as well as results about subgraphs, coronas and joins.  Section 3 focuses on two families of graphs: spiders and wheels.  We find exact values for $b_2(G)$ and $t_2(G)$ when $G$ is a spider or wheel and show that these parameters are within $1$ for spiders but can be arbitrarily far apart for wheels.  We turn to Cartesian products in Section 4 and provide examples and  prove upper bounds for both the $2$-burning number and the $2$-burning source number.  In our concluding section, we discuss connections between $b_2(G)$, $t_2(G)$, and the minimum size of an associated percolating set.  We also pose a few open questions.  

\section{Definitions and preliminary results}

We assume all graphs to be finite and connected.
\subsection{Definitions}

The following algorithm gives a precise description of the $2$-burning process on a graph with  a specified sequence of sources.  Note that the algorithm will always terminate since our graphs are finite.
 
 \smallskip
 
\noindent
{\bf Algorithm $2$-burning }
\smallskip

\noindent
{\bf Input:}  A graph $G$ and a sequence $(s_1, s_2, \ldots, s_m)$ of vertices of $G$ (called \emph{sources}).

\smallskip

\noindent
{\bf Output:}  Either returns the number of rounds until all vertices are blue, or reports that there will always be an uncolored vertex.
\smallskip

\noindent
{\bf Initialize:}  At round $0$, all vertices are uncolored.

\smallskip

\noindent
{\bf Loop:}  Iterate the following starting at $j=1$ and continue  until either all vertices are blue or no new vertices become blue.

At round $j$: 

 (i) If $j \le m$ and $s_j$ is uncolored, then color the source vertex $s_j$  blue.

(ii) For every uncolored vertex $v$ with at least two blue neighbors at round $j-1$, color $v$  blue.

If all vertices are blue, terminate and report $ j$, which is  the number of rounds until all vertices become blue.

If  $j > m$ and no vertices turned blue in round $j$, terminate and report that there will always be an uncolored vertex for this source sequence.
 
 Otherwise, increment $j$ and begin the loop again.
 \smallskip
 
 \noindent
 {\bf (End of Algorithm)}

\bigskip

With this formal description of the $2$-burning process, we can now introduce our fundamental definitions.

\begin{defn}{\rm
A  \emph{$2$-burning sequence} $s$ for a graph $G$ is a sequence $(s_1,s_2, s_3, \ldots, s_m)$ of vertices of $G$ for which all vertices are blue when Algorithm $2$-burning terminates. 
}
\end{defn}

\begin{defn}
{\rm If  $s = (s_1,s_2, s_3, \ldots, s_m)$ is a $2$-burning sequence for graph $G$ then  the length of $s$, denoted by $\len(s)$, is $m$ and  
$\rd(s)$ is the output of Algorithm $2$-burning, that is, the number of rounds until all vertices of $G$ are blue. 
}
\end{defn}

 Returning to the graph in Figure~\ref{fig:firstexample}, if $s=(a,b,c)$ then $\len(s)=3$ and $\rd(s)=3$; however, if $s'=(a,c,b)$ then $\len(s')=3$ and $\rd(s')=4$.
 
\begin{defn}
{\rm The  \emph{$2$-burning number} for graph $G$, denoted by $b_2(G)$, is the minimum value of  $\rd(s)$, taken over all $2$-burning sequences for $G$.   A $2$-burning sequence that achieves this minimum is  called \emph{optimal}.  The \emph{$2$-burning source number} for graph $G$ is the minimum length of an optimal $2$-burning sequence for $G$ and is denoted by $t_2(G)$.
 } 
 \end{defn}
 
 For example,  selecting any two distinct source vertices provides an optimal $2$-burning sequence of minimum length for the complete graph $K_n$, but when $n \ge 3$,  it takes one additional round for all vertices to turn blue.  Thus  $b_2(K_n) = 3$  and $t_2(K_n) = 2$ for $n \ge 3$.
By definition, $t_2(G) \le b_2(G)$ for all graphs $G$.

\subsection{Subgraphs and vertices of degree one and two}

For  many real-valued functions defined on graphs  (e.g., chromatic number, maximum degree,  girth) the function value decreases for induced subgraphs.  However, the same is not true for the $2$-burning number.  For example, the path $P_{7}$ is an induced subgraph of the path $P_{12}$ and also of the wheel $W_8$.    We will see in Theorems~\ref{path-cycle-thm}  and \ref{wheel-thm}   that $b_2(P_{7}) = 5$, $b_2(P_{12}) = 7$ and $b_2(W_8) = 4$,  and indeed these theorems show that the gap between the $2$-burning number of a graph and that of an induced subgraph can be made arbitrarily large in either direction using paths and wheels.   However, for \emph{spanning} subgraphs, we can prove an inequality.  

\begin{lemma} \label{sameverticessublem}
    If $G$ and $H$ are connected graphs and $H$ is a subgraph of $G$ with $|V(G)|=|V(H)|$, then $b_{2}(G) \leq b_{2}(H)$ and $t_{2}(G) \leq t_{2}(H)$. 
\end{lemma}

\begin{proof}
Since the edges of $H$ are present in $G$, any
   $2$-burning sequence for $H$  will also be a  $2$-burning sequence for $G$.    Moreover, using any $2$-burning sequence $s$ for $H$, the number of rounds until all vertices are blue in $H$ is at least as large as when $s$ is used as a $2$-burning sequence for $G$.  Thus $b_{2}(G) \leq b_{2}(H)$ and $t_{2}(G) \leq t_{2}(H)$.
   \end{proof}

We next consider the role of leaves (i.e., vertices of  degree one) and vertices of degree two when creating a $2$-burning sequence.  Observe that there are only two ways for any vertex to turn blue, either it is a source or it has two blue neighbors. Since leaves have only one neighbor, they must be sources, which we record as an observation below.  
In Lemma~\ref{adjdeg2lem} we show that for  adjacent vertices of degree two, at least one must be a source.

\begin{observation}\label{obs:deg1}
If $v$ is a leaf of graph $G$ and $s$ is a $2$-burning sequence for $G$ then $v \in s$.  Consequently, if graph $G$ has $k$ leaves, then $b_2(G) \ge k$.
\end{observation}

\begin{lemma} \label{adjdeg2lem}
    If $G$ is a graph and $v_{1}$ and $v_{2}$ are adjacent vertices in $G$ with $\deg(v_{1})=\deg(v_{2})=2$, then at least one of $v_{1}$ and $v_{2}$ must be a source in any $2$-burning sequence. 
\end{lemma}

\begin{proof}
    Consider  a $2$-burning sequence for $G$ and assume for a contradiction neither $v_{1}$ nor $v_{2}$ are sources. The vertex $v_{1}$ can only turn blue after its two neighbors are blue so $v_{1}$ turns blue after $v_{2}$. By symmetry, $v_{2}$ turns blue after $v_{1}$, which is a contradiction. 
\end{proof}

The next lemma is helpful in identifying graphs $G$ with $b_2(G) > t_2(G)$. 

\begin{lemma} \label{waiting-lem}
Let $G$ be a graph  and   let $s$   be  a $2$-burning sequence   for $G$. If every vertex of $s$ is adjacent to a degree two non-source vertex then $\rd(s) \ge \len(s) + 1$.

\end{lemma}

\begin{proof}
By definition we know $\rd(s) \ge \len(s) $ so suppose for a contradiction that $\rd(s) = \len(s) $.  Let $v$ be the last source vertex in $s$.    Then $v$ has a  degree two non-source neighbor  $x$, and $x$ cannot turn blue until at least one round after $v$ turns blue, a contradiction.
\end{proof}

In \cite{LQL2021} the authors find the $2$-burning number for paths and cycles using direct arguments.

\begin{theorem} {\rm \cite{LQL2021}  }
If $P_n$ is the path on $n$ vertices and $C_n$ is the cycle on $n$ vertices then 
$b_2(P_n) = b_2(C_n) = \lceil\frac{n}{2}\rceil + 1.$
\label{path-cycle-thm}
\end{theorem}

The lower bound  $b_2(C_n) \ge  \lceil\frac{n}{2}\rceil + 1$
follows from  combining the results in  Lemmas~\ref{adjdeg2lem} and \ref{waiting-lem} and equality holds because selecting every other vertex  (when $n$ is even) and one additional source vertex (when $n$ is odd) yields an optimal $2$-burning sequence of length   $\lceil\frac{n}{2}\rceil$.  The result for paths then follows using Lemma~\ref{sameverticessublem}.

As  above, $C_n$ has an optimal $2$-burning sequence of  length   $\lceil\frac{n}{2}\rceil$ for all $n \ge 3$.   The path 
 $P_n$  has an optimal $2$-burning sequence of  length $\lceil\frac{n}{2}\rceil + 1$ 
  when $n$ is even, and of length $\lceil\frac{n}{2}\rceil$  when $n$ is odd.  We record this below. 
 
\begin{observation}\label{paths-cycles-seq-length}
If $P_n$ is the path on $n$ vertices and $C_n$ is the cycle on $n$ vertices then
$t_2(P_n) = \lceil\frac{n}{2}\rceil + 1$ and $ t_2(C_n) = \lceil\frac{n}{2}\rceil$   when $n$ is even and $t_2(P_n) = t_2(C_n) = \lceil\frac{n}{2}\rceil $ when $n$ is odd.
\end{observation}

\subsection{Coronas and joins}

Observation~\ref{obs:deg1} states that for any $2$-burning sequence of a graph $G$, every leaf must be a source vertex.  This motivates us to next consider a family of graphs with many leaves.  For any graph $G$ on $n \geq 2$ vertices, the corona $G \circ K_1$ is the graph on $2n$ vertices  obtained by joining a leaf to each vertex of $G$.

\begin{proposition}
For any graph $G$ on $n \geq 1$ vertices, $b_2(G \circ K_1) = t_2(G \circ K_1) =n+1$. 
\end{proposition}

\begin{proof}
 Observation~\ref{obs:deg1} implies that each of the $n$ leaves of  $G \circ K_1$ must be a source vertex in a $2$-burning sequence for  $G \circ K_1$.  Each vertex of $G$ is   adjacent to  only one leaf of $G \circ K_1$, thus there must also be a source vertex from among the  vertices of $G$.  Therefore,
   $n+1 \le t_2(G \circ K_1) \le b_2(G \circ K_1) $.    

 It remains to show $ b_2(G \circ K_1) \le n+1$.  Let $v_1$ be any vertex of $G$ and label the rest of the vertices $v_2,\dots,v_n$, according to a breadth-first search rooted at $v_1$.   Let $T$ be the resulting breadth first search tree of $G$, so $v_1$ is the root of $T$ and if $v_i$ is the parent of $v_j$ in $T$ then $i < j$.     Label the leaves of $G \circ K_1$ as $v_1',v_2',\dots,v_n'$ so that $v_i$ is adjacent to $v_i'$ for $1 \leq i \leq n$ and let $s = (v_1,v_2',v_3',v_4',\dots,v_n',v_1')$.
  
  When we run Algorithm $2$-burning on graph $G \circ K_1$ and sequence $s$, we see that for $t: 2 \le t \le n$, vertex $v_t'$ is blue at round $t$ as is the parent of $v_t$ in $T$.  Thus  $v_t$ has two blue neighbors after round $t$ and will be blue by round $t+1$.  Therefore, all vertices of $G \circ K_1$ are blue by round $n+1$ and $s$ is a $2$-burning sequence of length $n+1$ for $G \circ K_1$, proving  that $ b_2(G \circ K_1) \le n+1$.\end{proof}  
   
 While the corona $G \circ K_1$ is a family of graphs with a  high number of leaves, we next consider a family on the other end of the spectrum: graph joins.  The \emph{join} of graphs $G$ and $H$ is the graph denoted $G \vee H$ that has vertex set $V(G) \cup V(H)$ and edge set $E(G) \cup E(H) \cup \{xy: x \in V(G), y \in V(H)\}$.

Consider graphs $G$ and $H$ with   $|V(G)| \ge 2$ and  $|V(H)| \ge 2$.  If we select $s_1, s_2 \in V(G)$ and $s_3 \in V(H)$, then the sequence $(s_1,s_2,s_3)$ is a $2$-burning sequence for $G \vee H$ because all vertices   in the subgraph induced by $V(H)$ turn blue  in round $3$ and  subsequently, all vertices  in the subgraph induced by $V(G)$ are blue by round 4.  We record this in the following. 
   
\begin{observation}\label{obs:join} Let $G$ and $H$ be graphs on $m \geq 2$ and $n \geq 2$ vertices, respectively.  Then $b_2(G \vee H) \leq 4$ and $t_2(G \vee H) \leq 3$.
\end{observation}

When we apply Observation~\ref{obs:join} to the graphs $G = \overline{K_n}$ and $H = \overline{K_m}$ we get an upper bound  of $4$ for the $2$-burning number of the complete bipartite graph $K_{m,n}$.  In \cite{LQL2021} the authors proved directly that $b_2(K_{m,n}) = 4$ if $n,m \ge 4$ and   $b_2(K_{m,n}) = 3$ if $m \in \{2,3\}$ or $n \in \{2,3\}$.

A graph $G$ with universal vertex $u$ is an example of a graph join: $G = K_1 \vee (G-u)$ and Observation~\ref{obs:join} can be applied.  However, for a graph $G$ with universal vertex $u$, we can also bound the $2$-burning number of $G$ by the $1$-burning number of $G-u$.  We will use this idea later in the proof of Theorem~\ref{wheel-thm}, but state a more general result next. For a graph $G$, a set $D \subseteq V(G)$ is \emph{dominating set} if every vertex of $G$ is either in   $D$ or has a neighbor in $D$.  We denote by $\gamma(G)$, the minimum cardinality of a dominating set on graph $G$.   

\begin{proposition}\label{prop:univ} Let $G$ be a graph and $D \subseteq V(G)$ be a dominating set of cardinality $\gamma(G)$.  Then $b_2(G) \leq \gamma(G)+b_1(G - D)$.  
\end{proposition} 

\begin{proof} Let $G$ be a graph and $D \subseteq V(G)$ be a dominating set of cardinality $\gamma(G)$.  We choose the vertices of $D$ (in any order) to be the first $\gamma(G)$ source vertices for $G$.  At the end of round $\gamma(G)$, every vertex in $G-D$ is adjacent to a blue vertex.  Thus, the process reduces to the $1$-burning process on $G-D$.  
\end{proof} 

\section{The classes of spiders and wheels}

In this section, we consider two classes of graphs:  spiders and wheels.  For each, we are able to find the exact value for the $2$-burning number and the $2$-burning source number.  These classes provide an interesting contrast between $b_2(G)$ and $t_2(G)$. 
For spiders these quantities differ by at most 1, while for wheels they can be arbitrarily far apart.

\subsection{Spider Graphs}
 
The spider graph $S_{n_{1},n_{2},..,n_{r}}$  consists of  a central vertex $v^{\ast}$ of degree $m$,  and paths of lengths $n_1, n_2, \ldots, n_r$  emanating from $v^{\ast}$. Thus if we remove $v^{*}$ from  $S_{n_{1},n_{2},..,n_{r}}$, the resulting graph consists of the paths $P_{n_{1}}, P_{n_{2}},..,P_{n_{r}}$.   Figure~\ref{fig-spider} shows four spider graphs.   Note that when $r=1$ or $r=2$ the spider graph is   a path graph, thus we assume $r \ge 3$.

Interestingly, although determining the $1$-burning number of a spider graph is NP-hard~\cite{BessyEtAl}, we next determine the $2$-burning number exactly.

\begin{theorem}
    Let  $G$ be the spider graph  $S_{n_{1},n_{2},..,n_{r}}$ where  $r \geq 3$ and $n=n_{1}+n_{2}+...+n_{r}+1$ and let $k$ be the number of $n_i$ that are odd.  Then $b_{2}(G)=\lceil\frac{n}{2}\rceil+1$ if $k \le 2$ and $b_2(G) = \frac{n+k-1}{2}$ if $k \ge 3$.
    \label{spider-thm}
\end{theorem}

\begin{proof}
    Let $G$ be the spider graph $S_{n_{1},n_{2},..,n_{r}}$ and $v^{*}$ the central vertex. As discussed above,  $G- v^{*}$ is the graph whose components are  the paths  with $n_i$ vertices for   $1 \leq i \leq r$.  We denote by $H_i$ the $i$th component of $G-v^{*}$, which is a  path with $n_i$ vertices.  Let $x_i$ be the vertex of  $H_i$ that is adjacent to $v^{\ast}$ and $y_i$ be the vertex of $H_i$ that is farthest from $v^{\ast}$. 
 We construct $2$-burning sequences for $G$  using four cases that depend on  the value of $k$.
 In each case below,  when $n_i$ is even we choose $y_i$ and every other vertex of $H_i$ starting at $y_i$ to be a source.  Thus we have $\frac{n_i}{2}$ sources from $H_i$ for the even paths.  We will select $\frac{n_i + 1}{2}$ sources from $H_i$ when $n_i$ is odd but the particular sources chosen from these paths will vary in different cases.
 
 \smallskip
 
 First consider the case $k=0$,   which is illustrated for  $S_{4,4,4,2,2}$  in Figure~\ref{fig-spider} (a).  In this case, $n$ is odd.    In addition to the sources chosen from the $H_i$,  we also select $v^{\ast}$ as a source.   So the number of sources is  $1 + \frac{1}{2}( n_1 + n_2 + \cdots + n_r) = 1 + \frac{n-1}{2} = \lceil\frac{n}{2}\rceil$.  Every non-source vertex of $G$ is adjacent to two of these sources, so every vertex turns blue by round $\lceil \frac{n}{2} \rceil +1$ and thus $b_{2}(G)  \le \lceil \frac{n}{2} \rceil +1$. In this case,   the sources can  be arranged in any order.

\smallskip
 
Next consider  the case $k=1$,   which is illustrated for $S_{3,4,4,2}$ in Figure~\ref{fig-spider}(b).  In this case, $n$ is even.   Without    loss of generality, let $H_{1}$ be the odd path.  If $n_1=1$ let $y_1$ be a source, and if $n_1 \ge 3$,   let $y_1$ and its neighbor  $z_1$ be sources, as well as every other vertex starting at $z_1$.   In addition to the sources chosen from the $H_i$,  we also select $v^{\ast}$ as a source.   So the number of sources is  $1 + \frac{1}{2}( 1 +n_1 + n_2 + \cdots + n_r) = 1 + \frac{n}{2} = \lceil\frac{n}{2}\rceil + 1$.   Arrange the sources  so that $v^{*}$ is the first and  $y_1$ is the last.  The   remaining sources can  appear  in any order.   
        Every  non-source vertex is adjacent to two sources,  so all vertices will turn blue and we have constructed a $2$-burning sequence for $G$.  Furthermore, the only neighbor of the last source $y_1$ is itself a source vertex, thus all vertices are blue by round $\lceil\frac{n}{2}\rceil + 1$, and hence 
         $b_{2}(G) \le    \lceil \frac{n}{2} \rceil +1$.
 
 \smallskip
Our third case is   $k=2$,   which is illustrated for $S_{5,3,4,2}$  in Figure~\ref{fig-spider} (c).  In this case, $n$ is odd.  
Without loss of generality,  let $H_{1}$ and $H_2$  be the odd paths, and for these paths (as well as the even ones)   let $y_i$ and and every other vertex starting at $y_i$ be a source. 
  This means that $x_1$ and $ x_2 $ will be chosen as sources.  In this case we do \emph{not} select $v^{\ast}$ to be a source, but we do  select $x_1$ and $x_2$  as the first  two sources  so $v^{\ast}$ will turn blue in round $3$.   The number of sources is $ \frac{1}{2}( 2 +n_1 + n_2 + \cdots + n_r) =  \frac{1}{2}(1+n) =   \lceil \frac{n}{2} \rceil$.  Since $k=2$ and $r \ge 3$, we know $n \ge 5$ so $ \lceil \frac{n}{2} \rceil \ge 3$  and there are a sufficient number of rounds for $v^\ast$ to turn blue.  After  all the sources turn blue, every non-source is adjacent to  two blue vertices, so in one additional round, all the vertices become blue.  Thus $b_{2}(G) \le    \lceil\frac{n}{2}\rceil+1$.
    
 \smallskip
Finally, we consider the case    $k \ge 3$,  which is illustrated for $S_{5,3,3,3,3,4}$ in Figure~\ref{fig-spider}(d).  Without loss of generality, assume that $n_i$  is odd for $1 \le i \le 3$. 
 For the odd path $H_1$,  if $n_1=1$ let $y_1$ be a source, and if $n_1 \ge 3$,   
  let $y_1$ and its neighbor  $z_1$ be sources, as well as every other vertex starting at $z_1$.  For the remaining odd paths (as well  as the even ones),     let $y_i$ and every other vertex starting at $y_i$ be a source.  This means that $x_2$ and $x_3$ will be sources.   In this case we do \emph{not} select $v^{\ast}$ to be a source, but we do  select $x_2$ and $x_3$  as the first  two sources  so $v^{\ast}$ will turn blue in round $3$ and  we select $y_1$ to be the last source.     The   remaining sources can   appear  in any order.   There are  $\frac{n_i}{2}$ sources for each $H_i$ when $n_i$ is even and  $\frac{1+n_i}{2}$ sources for each $H_i$ when $n_i$ is odd.  Thus   the number of source vertices is  $  \frac{1}{2}(k + n_1 + n_2 + \cdots + n_m) = \frac{1}{2}(n+k-1) $.  
   Note that this number is an integer since $k$ and $n$ have opposite parity.  
   Every  non-source vertex is  either adjacent to two sources or to $v^{\ast}$ and  one source,  so all vertices will turn blue and we have constructed a $2$-burning sequence for $G$.  Furthermore, the only neighbor of the last source $y_1$ is itself a source vertex, thus all vertices are blue by round $\frac{1}{2}(n+k-1) $, hence         $b_{2}(G) \le   \frac{1}{2}(n+k-1) $.
   
   \smallskip
   
 It remains to show that  $b_{2}(G) \ge    \lceil\frac{n}{2}\rceil+1$ for $k \le 2$ and  $b_{2}(G) \ge   \frac{1}{2}(n+k-1) $ for $k \ge 3$.
   We begin by  calculating the minimum number of source  vertices in $H_i$ in a $2$-burning sequence.
 Of the   vertices of $H_i$, vertex $y_i$ is a leaf of $G$ and the remaining $n_i-1$ vertices have degree 2 in $G$.  By Lemma~\ref{adjdeg2lem},  at least $\frac{n_{i}-2}{2}$ of these degree two vertices must be sources and by Observation~\ref{obs:deg1}, the leaf $y_i$ must also be a source. Thus there are at least $\frac{n_{i}}{2}$ vertices in $H_{i}$ that are sources.   Summing over all $i$, we get at least $  \lceil \frac{n-1}{2} \rceil $ sources from $G-v^{\ast}$ in any $2$-burning sequence of $G$. 
    
In the first case, $k=0$,  so $n$ is odd and  the number of sources from  $G-v^{\ast}$ in any $2$-burning sequence is at least $ \lceil \frac{n}{2} \rceil - 1$.  If none of the $x_i$ are sources, then $v^{\ast}$ must be a source, so there are at least $  \lceil \frac{n}{2} \rceil $ sources. However, in this instance, every non-source has degree $2$ and every source is adjacent to a non-source, so by Lemma~\ref{waiting-lem}, the $2$-burning sequence requires  $  \lceil \frac{n}{2} \rceil  + 1 $ rounds.  If exactly one of the $x_i$ are sources, then again  $v^{\ast}$ must be a source, so there are at least $  \lceil \frac{n}{2} \rceil + 1 $ sources.  Otherwise, at least two of the $x_i$ are sources and again there are at least $  \lceil \frac{n}{2} \rceil + 1 $ sources.   Thus any $2$-burning sequence requires at least $\lceil \frac{n}{2} \rceil  + 1 $ rounds and $b_{2}(G) \ge    \lceil\frac{n}{2}\rceil+1$.

   In the next case, $k=1$, so $n$ is even and the number of sources from  $G-v^{\ast}$ in any $2$-burning sequence of $G$ is at least   $  \lceil \frac{n}{2} \rceil $.  If only one $x_i$ is a source then $v^{\ast}$ must be a source, so in any case there are at least  $\lceil \frac{n}{2} \rceil  + 1$ sources.  Thus any $2$-burning sequence requires at least $\lceil \frac{n}{2} \rceil  + 1 $ rounds and $b_{2}(G) \ge    \lceil\frac{n}{2}\rceil+1$.
   
   In the case $k=2$ we again have $n$ odd and we may assume that the odd paths are $H_1$ and $H_2$.   In any $2$-burning sequence  there are at least $\frac{1 + n_i}{2} $ sources from the odd path $H_i$ for $i = 1,2$, so the total number of sources from  $G-v^{\ast}$  is it least $\frac{1}{2}(2 + n_1 + n_2 + \cdots + n_r) =  \lceil \frac{1+n}{2} \rceil  =  \lceil \frac{n}{2} \rceil $.     Suppose for a contradiction that there exists a $2$-burning sequence $s$ with $\len(s) = \rd(s) =  \lceil \frac{n}{2} \rceil.$  Consider the last source $s_k$ of $s$ (i.e., $k =  \lceil \frac{n}{2} \rceil $).  If $s_k = x_1$ or $s_k = x_2$ then $v^{\ast}$ does not turn blue until round $ \lceil \frac{n}{2} \rceil + 1$, a contradiction.  Otherwise, $s_k$ is adjacent to a degree $2$ non-source vertex and that vertex does not turn blue until round  $\lceil \frac{n}{2} \rceil  + 1$, a contradiction.   Thus any $2$-burning sequence requires at least $\lceil \frac{n}{2} \rceil  + 1 $ rounds and $b_{2}(G) \ge    \lceil\frac{n}{2}\rceil+1$.
   
   Finally, we consider the case $k \ge 3$.  In any $2$-burning sequence for $G$ there are 
   at least $\frac{n_i}{2}$ sources from each $H_i$ when $n_i$ is even and  at least $\frac{1+n_i}{2}$ sources from each $H_i$ when $n_i$ is odd.  Thus the total number of source vertices is at least $ \frac{1}{2}(k + n_1 + n_2 + \cdots + n_r) = \frac{1}{2}(k+n-1) $ and consequently $b_{2}(G) \ge \frac{1}{2}(k+n-1).$
   \end{proof}
   
 \begin{figure}[htbp]
 \[ \includegraphics[width=\textwidth]{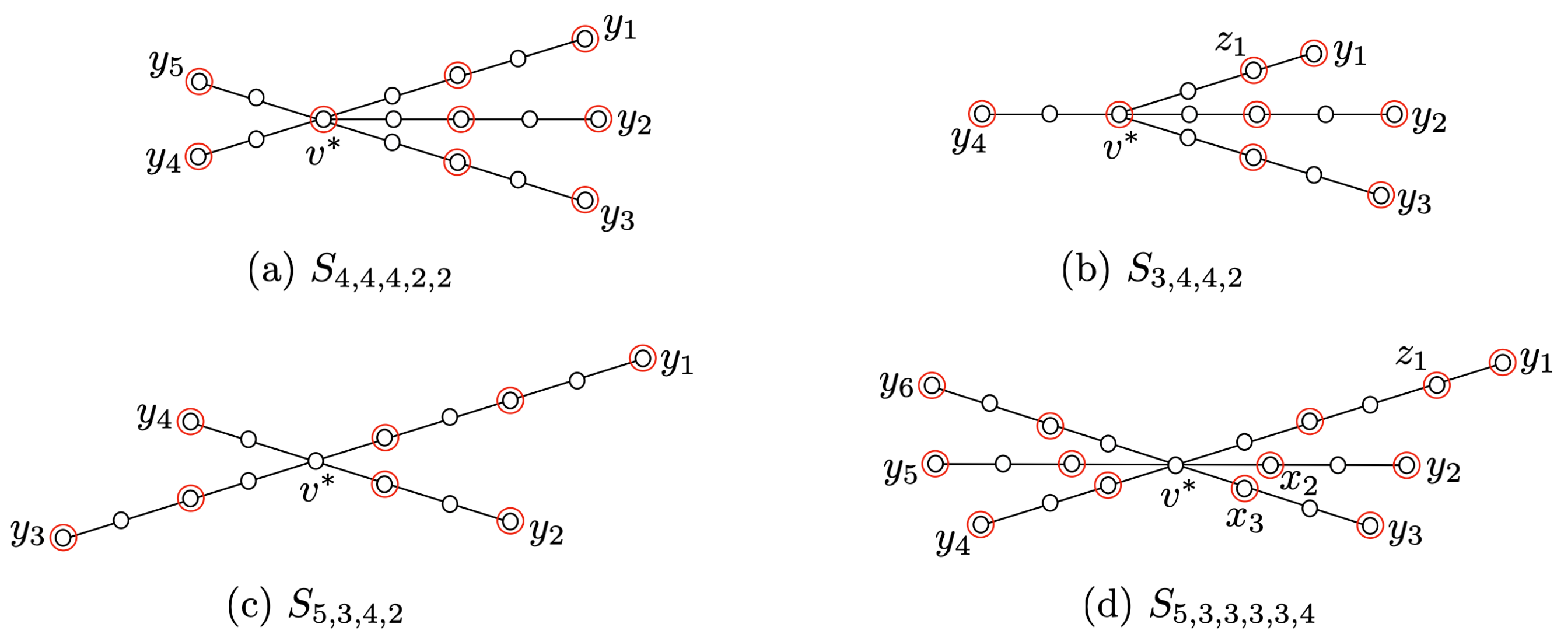}
 \]
   \caption{Four spider graphs in which the circled vertices are sources in an optimal $2$-burning sequence.}
   \label{fig-spider}
   \end{figure}
   
   \begin{corollary}
    Let  $G$ be the spider graph $S_{n_{1},n_{2},..,n_{r}}$ where $r \geq 3$ and $n=n_{1}+n_{2}+...+n_{r}+1$ and let $k$ be the number of $n_i$ that are odd.  If $k=0$ or $k=2$ then $t_2(G) = b_2(G) -1 $  and if  $k=1$ or $k \ge 3$ then $t_2(G) = b_2(G).$ 
\end{corollary}

\begin{proof}
In the proof of Theorem~\ref{spider-thm} we show that for $k=0$ and $k=2$, every $2$-burning sequence $s$ has $\len(s) \ge \lceil \frac{n}{2} \rceil$ and there exists a $2$-burning sequence of that length.  Thus when $k=0$ or $k=2$ we have $t_2(G) = b_2(G) - 1$.  When $k=1$ the proof shows that $\len(s) \ge  \lceil \frac{n}{2} \rceil + 1$ for any $2$-burning sequence $s$, hence $t_2(G) = b_2(G)$.   Finally, when $k \ge 3$ we show $\len(s) \ge \frac{1}{2}(n + k - 1)$ in the proof of Theorem~\ref{spider-thm}, so $t_2(G) = b_2(G)$.
\end{proof}
   
\subsection{Wheel Graphs}

The \emph{wheel graph} $W_{n}$ is a graph formed from the cycle $C_{n}$ by adding a central vertex
  that is adjacent to all the vertices of the cycle.  By Proposition~\ref{prop:univ}  we know,  
  $b_2(W_{n+1}) \leq 1+b_1(C_n) = 1+\lceil \sqrt{n} \ \rceil$; the latter equality due to~\cite{BJR2016}. 
   In  Theorem~\ref{wheel-thm} we determine $b_2(W_{n+1})$ exactly, and 
Figure~\ref{fig-wheel} shows the optimal $2$-burning sequences for $W_{30}$ and $W_{22}$ that are constructed in the proof.

\begin{theorem}
    If $n \geq 5$, then $b_{2}(W_{n})=\lceil \sqrt{n+6} \  \rceil$ and $b_{2}(W_{4})=3$.
    \label{wheel-thm}
\end{theorem}

\begin{proof}
Let $v^{*}$ be the central vertex of $W_n$ and label the vertices along the cycle consecutively as $1, 2, 3, \ldots, n$ where $n \ge 5$.
We consider  two possibilities for $2$-burning sequences for $W_n$: those for which $v^{*}$ is a source and those for which $v^{*}$ is not  a source.  Note that it is not useful to choose $v^{*}$ as a source in round 3 or later since it will be adjacent to the first two source vertices and will therefore turn blue at round 3.  Additionally, note that selecting $v^{*}$ as the source in round 2 is equivalent to selecting it as the source in round 1, so we consider the latter in Case 1.

\medskip

\noindent {Case 1:} The central vertex $v^{*}$ is the first source.  

Observe that  in this case, once the central vertex turns blue, each vertex in the cycle is adjacent to one blue vertex, namely $v^{*}$, so it will turn blue when it is chosen as a source or when it is adjacent to another blue vertex of the cycle.  This reduces the problem to finding  the $1$-burning  number for $C_n$, and adding one for the initial round of selecting $v^{*}$ as a source.  
          In \cite{BJR2016},  it is proven  that $b_{1}(C_{n})=\lceil \sqrt{n}  \ \rceil$, hence in Case  1, the minimum number of rounds for all vertices to turn blue is  $ 1 + \lceil \sqrt{n}  \ \rceil$.
  
  \medskip

\noindent {Case 2:} 
The central vertex $v^{*}$ is not a source vertex. 

Let $s = (s_1,s_2,\dots,s_m)$ be a $2$-burning sequence  for which $s_i \neq v^{\ast}$ for $1 \le i \le m$.  Let $k = \rd(s)$, so $3 \le m \le k$.  
Apply Algorithm $2$-burning with input $W_n$ and sequence $s$.    In round $3$ the central vertex  $v^*$ turns blue.  After that, each  vertex with a blue neighbor on the cycle  becomes blue in the next round. In round $4$, the neighbors of $s_1$  on the cycle turn blue (if they are not already blue) and in round $5$ their neighbors on the cycle turn blue, and this continues propagating outward along the cycle in both directions from $s_1$.  Thus at each round, starting at round $4$, vertex $s_1$ is responsible for at most two new blue vertices, so after $k$ rounds,  source $s_1$ turns at most $2(k-3)$ vertices of the cycle blue.  The same is true for sources $s_2$ and $s_3$.  For $j \ge 4$, source $s_j$ is responsible for turning at most $2$ of the cycle vertices blue in each round from $j+1$ to $k$, so it turns at most $2(k-j)$ vertices of the cycle blue in total.    Let $B$ be the set of blue vertices in the cycle after $k$ rounds, thus $|B| \le m+3 \cdot 2 (k-3)+\sum_{j=4}^{m} 2(k-j)$.

    We simplify this sum to 
     \begin{eqnarray*}
|B| &\le&m+6k-18 +2\sum_{j=4}^{m} k - 2\sum_{j=4}^{m} j  \\
 &=& m + 6k - 18 + 2k(m-3) - (m+4)(m-3)\\
 &=& 2km - m^2 -6 = m(2k-m) -6.
\end{eqnarray*}
There are $n$ vertices on the cycle, and all are blue after round $k$, thus $n \le m(2k-m) -6$.  The product $m(2k-m)$ is maximized when $k=m$, so $n  \le k^2-6$
        or equivalently $k \ge \sqrt{n+6}$.  We know $k$ is an integer, so in fact, $k \ge \lceil\sqrt{n+6} \  \rceil$.
           Hence when $v^{*}$ is not a source vertex, every $2$-burning sequence for $W_n$ requires at least  $ \lceil \sqrt{n+6} \  \rceil$ rounds  and at least $m$ source vertices where $m(2k-m) \ge n+6$.
        \smallskip
    
Next  we show that the quantity $ \lceil \sqrt{n+6} \  \rceil$  is also sufficient.  Let $k = \lceil \sqrt{n+6} \  \rceil$ so $k^2 \ge n+6$.  Choose $m$ so that  $2km - m^2 \ge n+6$, but $2kr - r^2 < n+6$ for $r < m$.   That is, $m$ is the minimum integer for which $m(2k - m) \ge n+6$.  Note that $m$ is well-defined since for $r=0$ we know $0 <  n+6$ and for $r=k$ we know $k^2 \ge n+6$.  
We construct a $2$-burning sequence $s = (s_1,s_2,\dots,s_m)$ and show that $\rd(s) \le k$.  This is illustrated in Figure~\ref{fig-wheel}
for $n=30$ (where $k=6$ and $m=6$)  and for  $n=26$ (where $k=6$ and $m=4$).
Select  the source vertices as follows.  Let $s_1 = 1$, \  $s_2 = s_1 + 2(k-3) + 1$,  and  $s_3 = s_2 + 2(k-3) + 1$.  For $ 4 \le j \le k$, let $s_j = s_{j-1} + (k-j+1) + (k-j) + 1$.   One can check that $m \le 3$ only when $n \le 9$, and in each of those cases our sequence $s$ is a $2$-burning sequence  for $W_n$.  Thus we may assume $m \ge 4$.

 By the end of round $3$, the source vertices $s_1, s_2, s_3$ are blue as is the central vertex $v^{\ast}$.  Once  $v^{\ast}$ is blue,  each uncolored vertex with a blue neighbor on the cycle becomes blue in the next round.  
First consider vertices $v$ with    $s_1 < v < s_2$.  By construction, there are $2(k-3)$ such vertices, so each is distance at most $k-3$ from  one of $s_1, s_2$ along the cycle.    Since  $s_1, s_2$ and $v^{\ast}$ are all blue by round 3, and there are $k$ rounds total, vertex $v$ will be blue after round $k$.  The same is true for $s_2 < v < s_3$.

Next consider vertices $v$ with $s_{j-1} < v < s_j$ for some $j$ with $4 \le j \le m$.   By construction, $s_j - s_{j-1} = 1 + (k-j+1) + (k-j)$, so there are  there are $(k-j) + (k-(j-1))$  vertices strictly between $s_j$ and $s_{j-1}$ along the cycle.  The $k-j$ vertices closest to $s_j$ all turn blue by round $k$ since $s_j$ turns blue at round $j$ and there are $k-j$ rounds remaining.  Similarly, the  $k-(j-1)$ vertices closest to $s_{j-1}$ all turn blue by round $k$.  Thus again, vertex $v$ will be blue after round $k$.  

Finally, we consider the vertices $v$  with $s_m < v < n$.  As before, the $k-3$ vertices closest to $s_1$ along the cycle are all blue after round $k$, so it suffices to consider $v$ with $s_m < v \le n-(k-3)$.
 Using our recursive formulas above, we can derive explicit formulas for the  source vertices.  In particular,  $s_1 = 1, $\ $  s_2 =  2 +  2(k-3),$\  $s_3 =  3+ 4(k-3),$ and 
  \begin{eqnarray*}
s_m &=&  m + 3(k-3) + (k-m) +  2 \sum_{i=3}^{m-1}(k-i)  \\
 &=& 4k-9 +  2 \sum_{i=3}^{m-1} k -  2 \sum_{i=3}^{m-1}i  \\
 &=&4k-9 + 2k(m-3)- (m+2)(m-3)\\
 &=& 2km - m^2 + m  -2k - 3.
\end{eqnarray*}

Hence,  $s_m + (k-m) = 2km - m^2  - k - 3  \ge n+6 - k - 3 = n - (k-3)$.
 Since $s_m$ turns blue in round $m$ and there are $k-m$ rounds remaining,   the vertices $v$ with $s_m < v \le n-(k-3)$
 will all be blue after round $k$.  
    Thus $s$ is a $2$-burning sequence for $W_n$ and $\rd(s) \le k =   \lceil \sqrt{n+6} \  \rceil$.  Hence we have shown that 
when $v^{\ast}$ is not a source, the minimum number of rounds for all vertices of $W_n$ to turn blue is   $\lceil \sqrt{n+6} \ \rceil $.

Combining the results of these cases, we conclude that  $b_2(W_n) = \min \{ 1 + \lceil \sqrt{n} \ \rceil, \lceil \sqrt{n+6} \ \rceil \}$.  By inspection, $b_{2}(W_{4})=3$, and it is not hard to show that $\lceil \sqrt{n+6} \  \rceil  \le 1 + \lceil \sqrt{n} \  \rceil  $  for $n \geq 5$,  completing the proof.
\end{proof}

In our proof of Theorem~\ref{wheel-thm}
we found the minimum number of sources needed for an optimal $2$-burning sequence for $W_n$. We record this in the following corollary.

\begin{corollary}
If $n\ge 5$ and $k = \lceil \sqrt{n+6} \ \rceil$ then $t_2(W_n)$ is the minimum integer $m \ge 3$  for which $m(2k-m) \ge n+6$.
\label{wheel-cor}
\end{corollary}

  \begin{figure}
  \[ \includegraphics[width=\textwidth]{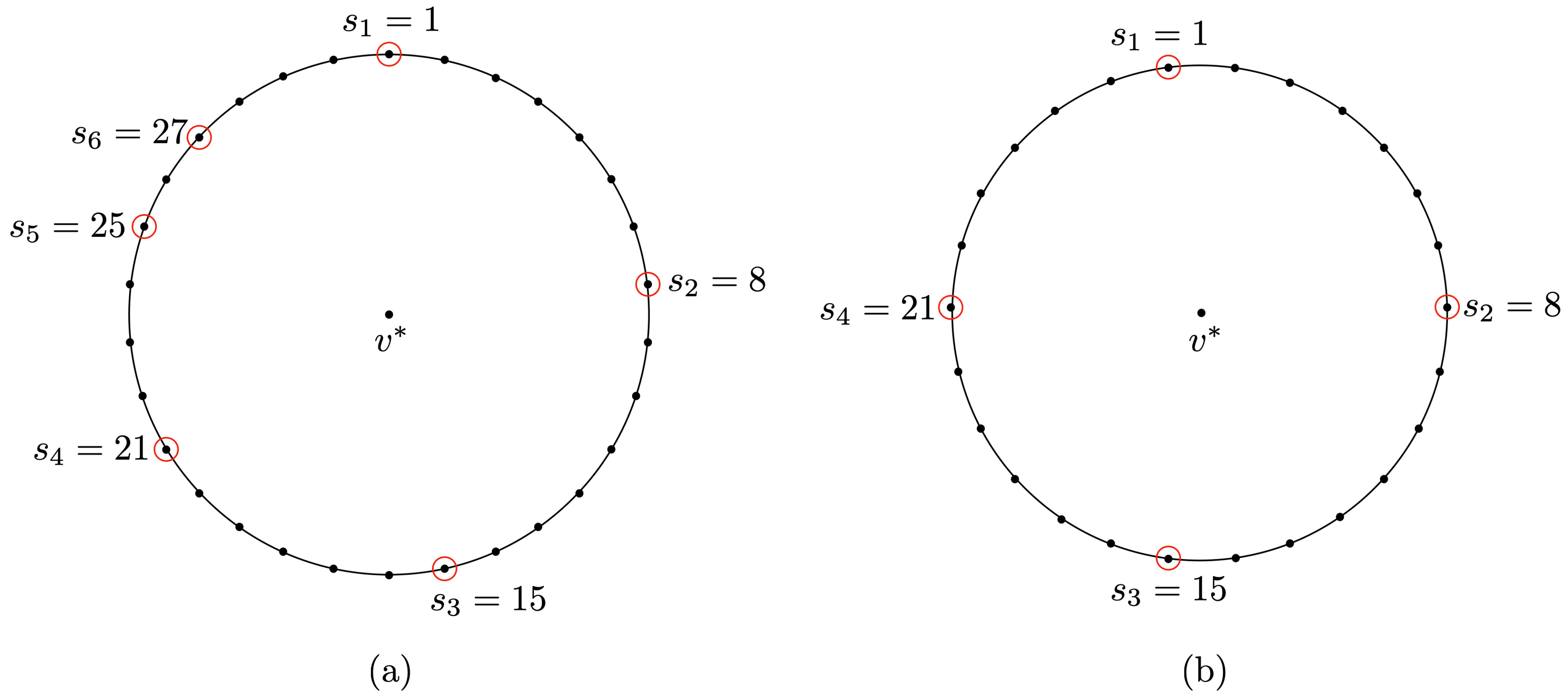} \]
   \caption{The wheels $W_{30}$  and $W_{26}$ where the edges incident to $v^{*}$ are omitted and only the sources in our  optimal $2$-burning sequences  are labeled.}
   \label{fig-wheel}
   \end{figure}

The wheel graphs  $W_{30}$ and  $W_{26}$ and optimal $2$-burning sequences for them are illustrated in  Figure~\ref{fig-wheel}.  For $W_{30}$
the optimal $2$-burning sequence constructed in the proof of Theorem~\ref{wheel-thm} is  $s=(1, 8, 15, 21, 25, 27)$ where $k= b_2(W_{30})=6$ and $m =t_2(W_{30} )= 6$.  For $W_{26}$
the optimal $2$-burning sequence constructed in the proof of Theorem~\ref{wheel-thm} is  $s=(1, 8, 15, 21)$ where $k= b_2(W_{26})=6$ and $m =t_2(W_{26}) = 4$.

  The next theorem shows that there exist wheel graphs $G$ for which  the $2$-burning number is arbitrarily larger than the length of an optimal $2$-burning sequence.
 
  \begin{theorem}\label{thm:wheel_R}
  For any  integer $r$, there exists an integer $n$ for which $b_2(W_n) - t_2(W_n) \ge r$.
  \end{theorem}
  
 \begin{proof}
  Fix an integer $r$ and choose an integer $k \ge 5$ such that $2(k-1) > r^2$.   Let  $n = (k-1)^2 -5$, so $n+ 6 = (k-1)^2 + 1 = k^2 - 2(k-1)$, and thus $\lceil \sqrt{n+6}\ \rceil = k$.
  By Theorem~\ref{wheel-thm}  we know that  $ b_2(W_{n}) = k$ and  by  Corollary~\ref{wheel-cor}  we know that $t_2(W_n)$ is the smallest integer $m \ge 3$ for which $m(2k-m) \ge n+6. $ Observe that for $j = k-r$ we have
  $$j(2k-j) = (k-r)(k+r) = k^2 - r^2  > k^2 - 2(k-1) = n+6.$$
  Hence $m \le k-r$ and $b_2(W_n) - t_2(W_n) = k- m \ge r$ as desired.  
    \end{proof}

\section{Cartesian products}\label{sec:cart}

\label{cartesian-sec}
Motivated by the interesting relationships that often emerge between graph parameters of input graphs versus their product, we study the $2$-burning number and the $2$-burning source number of Cartesian products.
The \emph{Cartesian product} of graphs $G$ and $H$ is the graph denoted $G \x H$ that has vertex set $\{(u,v): u \in V(G), v \in V(H) \}$ and  $(u_1,v_1) (u_2,v_2) \in E(G \x H) $ if and only if one of the following hold:  (i)  $u_1=u_2$ and $v_1v_2 \in E(H)$ or  (ii)  $v_1=v_2$ and $u_1u_2 \in E(G)$. 

 Recall that Observation~\ref{obs:deg1} provided a lower bound for the $2$-burning number of a graph based on the number of leaves; namely, if graph $G$ has $k$ leaves, then $b_2(G) \geq k$.  Cartesian products of connected graphs have no leaves, so we begin by determining $b_2(K_m \x K_n)$, and use this to  provide a lower bound for the $2$-burning number of the Cartesian product of any pair of connected graphs on $m$ and $n$ vertices.

\begin{theorem}\label{thm:KmKn} If $m \geq 5$ and $n = 3$, or $m \geq n \geq 4$, then $t_2(K_m \x K_n) = 2$ and  $b_2(K_m \x K_n) = 5$.\end{theorem}

\begin{proof} Let $V(K_m) = \{u_1,u_2,\dots,u_m\}$ and $V(K_n) = \{v_1,v_2,\dots,v_n\}$.  Partition $K_m \x K_n$ into $m$ copies of $K_n$, labeled $H_1,H_2,\dots,H_m$ where $u_i$ is the first coordinate of each vertex in $H_i$.  
For the lower bound on $b_2(K_m \x K_n)$, without loss of generality, suppose the first source vertex is $(u_1,v_1)$ in $H_1$.  We consider two cases.\smallskip

\noindent Case 1:  The second source vertex is in $H_1$. 

At the end of round $3$, every vertex in $H_1$ has turned blue.  Without loss of generality, the third source vertex is in $H_2$.  So at the end of round $3$, there are at least three uncolored vertices in $H_2$ and every vertex in $H_3$ is uncolored.  Let $(u_2,v_j)$ and $(u_2,v_k)$ be distinct vertices of $H_2$ that are uncolored at the end of round $3$, where $1 \leq j,k \leq n$ and $j \neq k$.  Then at most one of $(u_3,v_j),(u_3,v_k)$ in $H_3$ turns blue during round $4$, leaving at least one uncolored at the end of round $4$.\smallskip

\noindent Case 2: The second source vertex is not in $H_1$.

Without loss of generality, assume the second source is $(u_2,v_j)$ in $H_2$.  If $j=1$ then during  round $3$, exactly one non-source vertex turns blue in  $H_i$ for $3 \le i \le m$, namely  $(u_i,v_1)$.  Without loss of generality, assume $(u_1,v_2)$ is the third source vertex. Then during round $4$, exactly one non-source vertex   turns blue in each  $H_i$ for $3 \le i \le m$, namely  $(u_i,v_2)$.  Since $|V(H_i)| \geq 3$ for $i: 1 \leq i \leq m$ and at most one vertex from $H_1 \cup H_2$ is a source vertex in round $4$, there is at least one uncolored vertex amongst the vertices of $H_1 \cup H_2$ at the end of round $4$. 
    
If $j \neq 1$, without loss of generality we may assume $j=2$, so the first two sources are $(u_1,v_1)$ and $(u_2,v_2)$.  Then the only non-source vertices to turn blue in round $3$ are $(u_1,v_2)$ and $(u_2,v_1)$.  During round $4$, the remaining vertices of $H_1 \cup H_2$ turn blue; along with at most two non-source vertices in each of $H_3,H_4$.  At the end of round $4$, there are at most two source vertices in $H_3 \cup H_4$.  If $m \geq n \geq 4$, this leaves at least two vertices in $H_3 \cup H_4$ uncolored by the end of round 4.  If $m \geq 5$ and $n=3$, observe that during round $4$, exactly two non-source vertices in each of $H_3,H_4,H_5$ turn blue.  Then at the end of round $4$, there are at most two source vertices in $H_3 \cup H_4 \cup H_5$, leaving at least one uncolored vertex in $H_3 \cup H_4 \cup H_5$.\medskip

The above two cases show that $b_2(K_m \x K_n) \geq 5$.  To see that $b_2(K_m \x K_n) = 5$, let $s = ((u_1,v_1),(u_2,v_2))$ and observe that $s$ is a $2$-burning sequence of $K_m \x K_n$ that results in every vertex blue by the end of round $5$. It also shows that $t_2(K_m \x K_n) \leq 2$.  Any graph with two  or more vertices requires at least two sources, so $t_2(K_m \x K_n) = 2$.\end{proof} 

If $G$ and $H$ are graphs with $m =|V(G)|$ and $n = |V(H)|$ then $G \x H$ is a spanning subgraph of $K_m \x K_n$.  We combine Lemma~\ref{sameverticessublem} and Theorem~\ref{thm:KmKn} to conclude the following.

\begin{corollary}\label{cor:low} Let $G$ and $H$ be graphs on $m$ and $n$ vertices, respectively, where $m \geq 5$ and $n = 3$, or $m \geq n \geq 4$.  Then $t_2(G \x H) \geq 2$ and $b_2(G \x H) \geq 5.$ \end{corollary}

In the next theorem we show that $b_2(G)$ is an additional lower bound for $b_2(G \x H)$.

\begin{theorem}\label{thm:lower} Let $G$ and $H$ be connected graphs on $m$ and $n$ vertices, respectively, where $m \geq 5$ and $n=3$, or $m \geq n \geq 4$.  Then $b_2(G \x H) \geq \max \{5,b_2(G),b_2(H)\}$. \end{theorem}

\begin{proof} Let $s = \big( (u_1,v_1),(u_2,v_2),\dots,(u_k,v_k)\big)$ be an optimal $2$-burning sequence for $G \x H$. We note that the first coordinates of vertices in sequence $s$ may not all be distinct: for $i \neq j$, it is possible that $u_i = u_j$.  Similarly, the second coordinates of vertices in $s$ may not all be distinct.

For graph $G$, let $s_G = (u_1,u_2,\dots,u_k)$.   To complete the proof, we must show that when Algorithm 2-burning is applied to graph $G$ and sequence $s_G$, it will terminate when all vertices of $G$ are blue.  The fact that the vertices in $s_G$ are not necessarily all distinct causes no problem with this implementation: if source vertex $u_j$ turns blue before round $j$, then no source vertex turns blue during round $j$.  

We claim that any vertex $u_p \in V(G)$ turns blue by round $r$ (via $s_G$) if there exists a vertex $(u_p,v_q) \in V(G \x H)$ that turns blue by round $r$ (via $s$).  By construction of $s_G$, the claim is true when $u_p$ is in $s_G$.    For a contradiction suppose the claim is false and let $\ell$ be minimum so that there exists  $(u_p,v_q) \in V(G \x H)$ that turns blue by round $\ell$ (via $s$) but for which $u_p$ is not blue (via $s_G$) by round $\ell$.  
 For $(u_p,v_q)$ to turn blue during round $\ell$, it must have at least two neighbors that are blue by the end of round $\ell-1$.  We consider two cases for neighbors of $(u_p,v_q)$.

For the first case, assume $(u_p,v_q)$ has a neighbor $(u_p,v_c)$ that is blue by round $\ell-1$.  By the minimality of $\ell$, since $(u_p,v_c)$   is blue by round $\ell-1$, vertex $u_p$ in $G$ is blue by round $\ell-1$, a contradiction.

For the second case, assume $(u_a,v_q)$ and $(u_b,v_q)$ are distinct neighbors of $(u_p,v_q)$ and are both blue by round $\ell-1$.  Then $u_a,u_b \in N_G(u_p)$ and since $\ell$ is minimum, $u_a,u_b$ are both blue by  round $\ell-1$.  Therefore, $u_p$ is blue by the end of round $\ell$, a contradiction.  This proves our claim.

As a consequence, if $s$ results in every vertex of $G \x H$ blue by round $b_2(G \x H)$, then $s_G$ results in every vertex of $G$ blue by round $b_2(G \x H)$.  It follows that $b_2(G) \leq b_2(G \x H)$. A similar argument shows $b_2(H) \leq b_2(G \x H)$ and  $5 \leq b_2(G \x H)$ follows from Corollary~\ref{cor:low}.\end{proof}

 The proof of Theorem~\ref{thm:lower} maps source vertices in $G \x H$ to source vertices in $G$ to show $b_2(G) \leq b_2(G \x H)$.  Unfortunately, this approach cannot be used to provide an analogous bound for $t_2(G \x H)$ based on $t_2(G)$.  To see this, consider $P_4 \x P_3$, as shown in Figure~\ref{fig:p4p3}.  The sequence $s = \{(1,1),(2,2),(3,3),(4,2),(4,1)\}$ results in every vertex blue by the end of round $5$ and Theorem~\ref{thm:lower} implies that $b_2(P_4 \x P_3)\ge 5$.  Thus $b_2(P_4 \x P_3)=5$.  However, mapping the first three source vertices of $s$ to $P_4$ results in sequence $s_{P_4} = \{1,2,3\}$, which is not a $2$-burning sequence for $P_4$ as non-source vertex $4$ has only one neighbor and cannot ever turn blue.  We do not know if $t_2(G)$ forms a lower bound for $t_2(G \x H)$ in general.

\begin{figure}[htbp] 
\[ \includegraphics[width=0.625\textwidth]{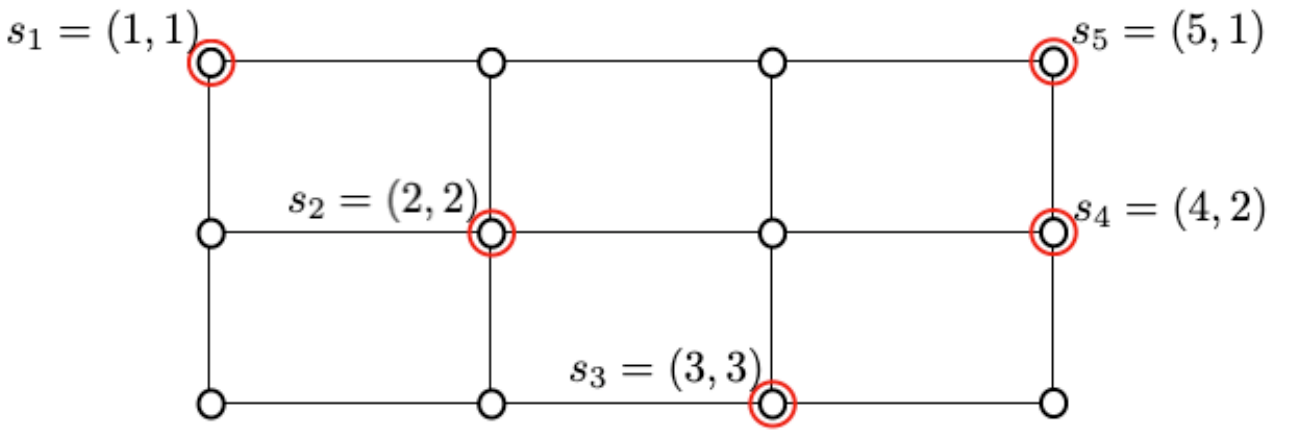} \]
\caption{The graph $P_4 \x P_3$ with $2$-burning sequence $s_1,s_2, s_3, s_4,s_5$  labeled.} 
\label{fig:p4p3}
\end{figure}

We next provide  upper bounds for the $2$-burning number and the  $2$-burning source number of the Cartesian product of graphs.  Figure~\ref{fig:grid} illustrates the proof in the case $G = P_5$ and $H= P_4$.

\begin{theorem}\label{thm:genupper-new} If $G$ and $H$ are graphs then  $t_2(G \x H) \leq t_2(G)t_2(H)$ and   $b_2(G \x H) \leq t_2(G)t_2(H) + (b_2(H) - t_2(H)) + (b_2(G) - t_2(G))$.
\end{theorem}

\begin{proof} Let $m = |V(G)|$ and label the vertices of $G$ as $g_1,g_2,\dots,g_m$ such that $g=(g_1,g_2,\dots,g_k)$ is an optimal $2$-burning sequence for $G$ with  $\len(g)=k=t_2(G)$.  Let   $n = |V(H)|$ and label the vertices of $H$ as $h_1, h_2, \ldots, h_n$ such that $h=(h_1,h_2,\dots,h_\ell)$ is an optimal $2$-burning sequence for $H$ with $\len(h)=\ell=t_2(H)$.  This is illustrated in Figure~\ref{fig:grid} where $G = P_5$, $H = P_4$, $m  = 5$, $k = 3$, $n = 4$, and $\ell = 3$.

For $i: 1 \le i \le k$, let $a_i$ be the sequence $(g_i,h_1),(g_i,h_2),\dots,(g_i,h_\ell)$ and let $s$ be the sequence $a_1,a_2, \ldots, a_k$.
Thus $s$ is a sequence of vertices in $G  \x  H$ and $\len(s) = k\ell = t_2(G)t_2(H)$.   In Figure~\ref{fig:grid}, the nine vertices of $s$ are circled and labeled $s_1,\dots,s_9$.  We will show that $s$ is a $2$-burning sequence for $G  \x  H$ and calculate the number of rounds until all vertices are blue.

For each $i:  1 \le i \le k$,  let $ H_i = \{(g_i,h_1),(g_i,h_2),\dots,(g_i,h_\ell), \ldots , (g_i,h_n)\}$ and note that   $H_i$  induces a copy of $H$ in  $G  \x  H$.  The first $\ell$  vertices listed in $H_i$  are blue after round $t_2(G)t_2(H)$ because they are part of $s$.  Since $(h_1,h_2, \ldots, h_{\ell})$ is a $2$-burning sequence for $H$, after an additional $b_2(H) - t_2(H)$ rounds, the remaining vertices in $H_i$ are blue.   Thus, if $1 \le i \le k$ and $1 \le j \le n$ then the vertex $(g_i,h_j)$ is blue after round $t_2(G)t_2(H) + b_2(H) - t_2(H)$.  

It remains to consider vertices of the form $(g_i,h_j)$ where $k+1 \le i \le m$ and $1 \le j \le n$.  For each $j:   1 \le j \le n$, let $G_j = \{
  (g_1,h_j),(g_2,h_j),\dots,(g_k,h_j), \ldots , (g_m,h_j)\}$ and note that  $G_j$ induces a copy of $G$ in  $G  \x  H$. 
    Since  $(g_1,g_2, \ldots, g_k)$ is a $2$-burning sequence for $G$   of length $t_2(G)$, all vertices of $G$ are blue  $b_2(G) - t_2(G)$ rounds after $g_k$ becomes blue.  Similarly, since $  (g_1,h_j),(g_2,h_j),\dots,(g_k,h_j)$ is a $2$-burning sequence of length $t_2(G)$ for  the copy of $G$   induced by $G_j$, all vertices of $G_j$ are blue  $b_2(G) - t_2(G)$ rounds after $(g_k,h_j)$ becomes blue.  We concluded above that  $(g_k,h_j)$ is blue  after round $t_2(G)t_2(H) + b_2(H) - t_2(H)$ for $1 \le j \le n$, so all vertices in $G_j$ are blue after round $t_2(G)t_2(H) + (b_2(H) - t_2(H)) + (b_2(G) - t_2(G))$.
     Thus sequence $s$ is a $2$-burning sequence for $G  \x  H$ of length $t_2(G)t_2(H)$ and $b_2(G  \x  H) \le  t_2(G)t_2(H) + (b_2(H) - t_2(H)) + (b_2(G) - t_2(G))$.  
\end{proof}

\begin{figure}[htbp]
\[ \includegraphics[width=0.9\textwidth]{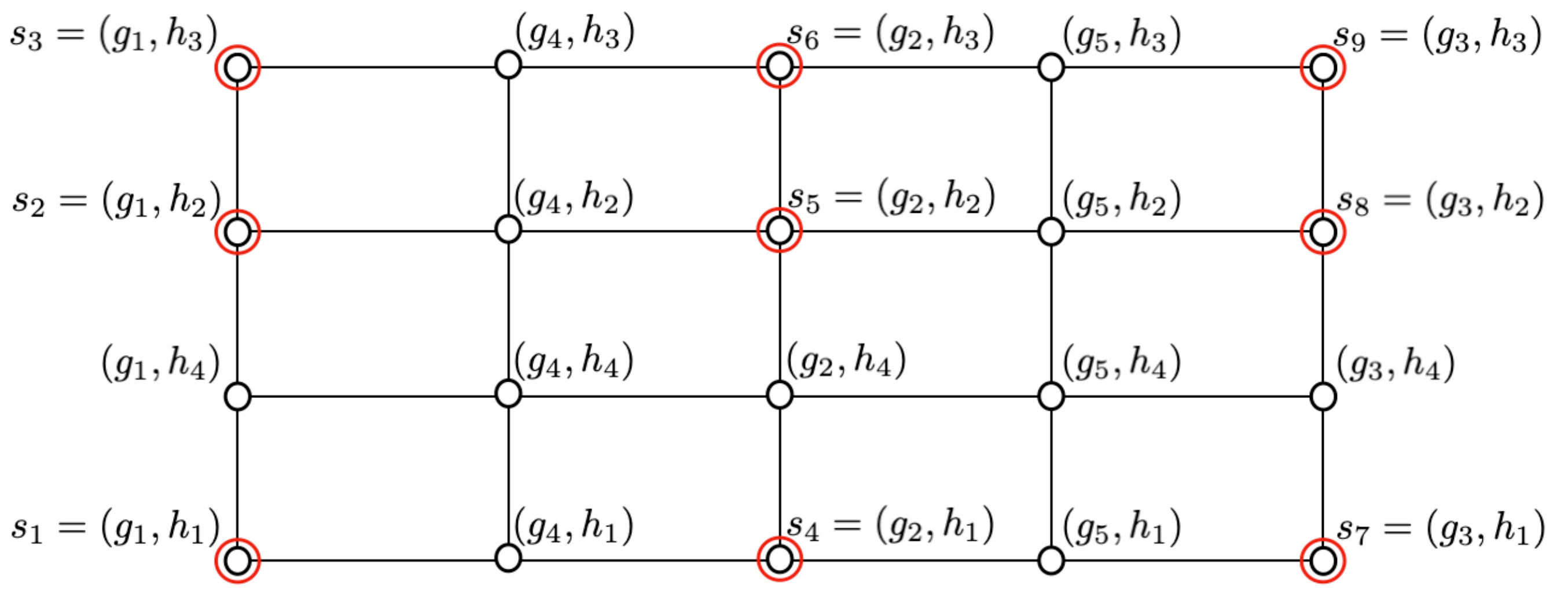} \]
\caption{The graph $P_5 \x P_4$ with vertices labeled as in the proof of Theorem~\ref{thm:genupper-new} and with source vertices $s_1, s_2, \ldots, s_9$.}
\label{fig:grid}
\end{figure}

The next result illustrates that for some graphs $G$ and $H$, the upper bound for $t_2(G \x H)$ given in Theorem~\ref{thm:genupper-new} is exact.

\begin{theorem} The graph $C_4 \x C_4$ has
 $b_2(C_4 \x C_4) = 5$ and $t_2(C_4 \x C_4)=4$.
 \label{Thm:C4C4}
  \end{theorem}

\begin{proof} Let $V(C_4) = \{1,2,3,4\}$ where vertex $i$ is adjacent to vertex $i+1$ (mod 4).  
Corollary~\ref{cor:low} provides the lower bound $b_2(C_4 \x C_4) \geq 5$; and it is straightforward to verify that $b_2(C_4 \x C_4) \le 5$ by choosing source vertices $(1,1),(2,2),(3,3),(4,4)$.  Thus $b_2(C_4 \x C_4) = 5$ and every optimal $2$-burning sequence $s$ has $\rd(s) = 5$.

Theorem~\ref{thm:genupper-new} provides the upper bound $t_2(C_4 \x C_4) \leq 4$.  For a contradiction, suppose $t_2(C_4 \x C_4) \le 3$ and fix an optimal $2$-burning sequence $s$ for $C_4 \x C_4$ with   $\len(s) \le 3$. Since $s$ is optimal, we know $\rd(s) = 5$ and since there is no $2$-burning sequence for $C_4 \x C_4$ with only $2$ source vertices, we also know $\len(s) = 3$.     We will use the terms ``rows'' and ``columns'' as illustrated in Figure~\ref{fig:C4C4} where ``row $i$" refers to the set $\{(i,1),(i,2),(i,3),(i,4)\}$;  and ``column $i$" refers to $\{(1,i),(2,i),(3,i),(4,i)\}$.\smallskip

 First consider the case in which there is at most one source vertex in each row and in each column.  Without loss of generality we may assume that  row $4$ and column $4$ contain no source vertices.  The vertices $(1,1)$ and $(1,3)$ cannot both be blue after round $3$ since at most one is a source and it would require a second source from row $1$ to turn the other blue by round $3$.  Thus the earliest $(1,1)$ and $(1,3)$ are both blue is round $4$, and therefore $(1,4)$ cannot turn blue until round $5$.  Similarly, vertices $(3,4)$, $(4,1)$ and $(4,3)$ cannot turn blue until round $5$. 
 So none of the neighbors of  vertex $(4,4)$  are blue until round $5$, and consequently, vertex $(4,4)$ is not blue at the end of round $5$, a contradiction. \smallskip
  
Otherwise, there exists a row or column containing two source vertices.  Without loss of generality we may assume that column $2$ has two source vertices.  There must be a source vertex in column $4$, or else no vertex in column $4$ will ever turn blue.  Without loss of generality we may assume that $(2,4)$ is a source vertex.  Any non-source vertex that turns blue in round $3$ must have two source neighbors.   Thus, at the end of round $3$, the  only  vertices that are blue, are in row $2$ and in column $2$.  Even if all seven of these vertices were blue at round $3$ (see Figure~\ref{fig:C4C4}), vertex $(4,4)$ would not turn blue by the end of round $5$, a contradiction.  Thus $t_2(C_4 \x C_4)=4$.
    \end{proof}

\begin{figure}[htbp] 
\[ \includegraphics[width=0.35\textwidth]{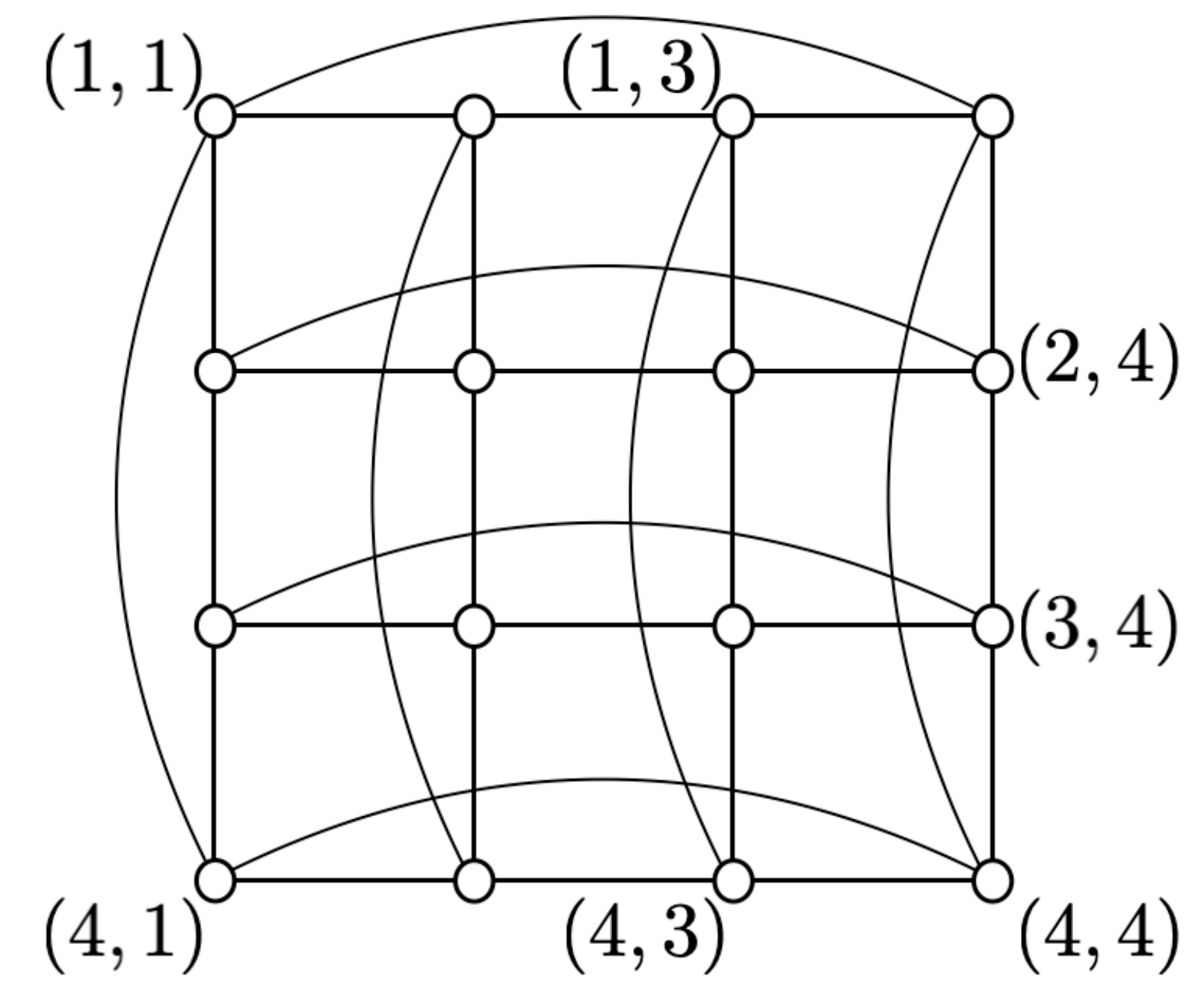} \]
\caption{The graph $C_4 \x C_4$ with some vertices labeled for reference, as in the proof of Theorem~\ref{Thm:C4C4}.}

\label{fig:C4C4}
\end{figure}

We know that $t_2(C_4) = 2$ and in Theorem~\ref{Thm:C4C4} we show
  $t_2(C_4 \x C_4) = 4$, so  $t_2(C_4 \x C_4) =  t_2(C_4)t_2(C_4)$.  Thus    there exist graphs $G$ and $H$ for which the bound $t_2(G \x H) \leq t_2(G)t_2(H)$ in Theorem~\ref{thm:genupper-new}
   is exact.  For the $2$-burning number, Theorem~\ref{thm:genupper-new} provides the bound $b_2(C_4 \x C_4) \leq 6$, but Theorem~\ref{Thm:C4C4}  shows that $b_2(C_4 \x C_4)=5$.    We do not know if there exist graphs $G$ and $H$ for which the bound on $b_2(G \x H)$ from Theorem~\ref{thm:genupper-new} is exact.  

\section{Conclusion and Directions}\label{sec:questions}

Since the $2$-burning numbers are known for $P_n$ and $C_n$, it is natural to consider the $2$-burning number of their products.  Theorems~\ref{thm:lower} and~\ref{thm:genupper-new} provide lower and upper bounds for $2$-burning numbers, but the exact values remain unknown.  For $P_n \x P_n$, we can gain insight and bounds on the $2$-burning number from bootstrap percolation since $P_n \x P_n$ is one of the few graphs for which bootstrap percolation has been studied from an extremal perspective.

In Section~\ref{intro-sec} we defined the related problem of $r$-neighbor bootstrap percolation where rather than having a sequence of sources turning blue in successive rounds, there is a set of sources that all turn blue in round $0$.  
 Most work on bootstrap percolation has focused on the process when the  source vertices  are chosen independently at random with probability $p$, but some researchers have considered the  problem of determining the minimum cardinality of a percolating set.  Following \cite{MN}, we denote by $m(G,r)$ the cardinality of a minimum percolating set on graph $G$ following the $r$-neighbor bootstrap process.  We denote by $\tau(G,r)$, the minimum number of rounds by which every vertex of $G$ is blue, over all possible minimum percolating sets on $G$.    Though we state the next observation for $r=2$, we note that it holds for general $r$ where the parameter $t_r(G)$ is defined analogously to $t_2(G)$. 

\begin{observation}\label{obs:easybound} For a graph $G$ on $n$ vertices $$m(G,2) \leq t_2(G) \leq b_2(G) \leq m(G,2)+\tau(G,2).$$ 
\end{observation}  

Observe that if $G = K_n$ and $n \geq 3$ then $m(K_n,2)=2 = t_2(K_n)$ and $b_2(K_n)=3$, and  $\tau(K_n,2)=1$; therefore the first and third inequalities are tight.  We have seen other examples (e.g., paths and cycles on an even number of vertices) that make the middle inequality tight. Since $\tau(G,2) > 0$ for all graphs $G$, there are no graphs for which equality holds throughout.

In Section~\ref{cartesian-sec} we considered Cartesian products and these have also been studied in the context of percolation.  
For $2$-neighbor bootstrap percolation, the extremal problem of determining the smallest percolating set was first considered Pete~\cite{Pete1997}, published in Hungarian; and later communicated by Balogh and Pete~\cite{BaloghPete}.
The vertices on the diagonal of the square grid  $P_n \x P_n$ constitute a percolating set for this graph, thus   $m(P_n \x P_n,2) \le n$.  Indeed it is shown in  \cite{Pete1997} and \cite{BaloghPete}  that $m(P_n \x P_n,2) = n$.   The same set of sources form a $2$-burning sequence for $P_n \x P_n$ and using this sequence, all vertices turn blue by round $2n$.
  Thus, $n \leq b_2(P_n \x P_n) \leq 2n$.  It would be interesting to determine the exact value of $b_2(P_n \x P_n)$ for general $n$.
  
\begin{question} Can we determine $b_2(P_n \x P_n)$ for all $n$? \end{question}

In this paper we have focused on  $b_2(G)$, the $2$-burning number, and our new parameter $t_2(G)$, the $2$-burning source number.
The quantities $b_2(G)$ and $t_2(G)$ are equal precisely when every optimal $2$-burning sequence requires a source in each round.  
We  found infinite families of graphs $G$ for which $t_2(G) = b_2(G)$; including subsets of paths, cycles, and spiders.  More generally, we would like to determine the set of graphs for which 
 these parameters are equal.

\begin{question} Can we characterize all graphs $G$ for which $t_2(G)=b_2(G)$? \end{question}

Finally, we state two questions relating to the discussions succeeding Theorems~\ref{thm:lower} and~\ref{Thm:C4C4}.

\begin{question} For arbitrary graphs $G$ and $H$, are $t_2(G)$ and $t_2(H)$ lower bounds for $t_2(G \x H)$?
\end{question}

\begin{question}\label{q1} Do there exist graphs $G$ and $H$ for which the bound on $b_2(G \x H)$ from Theorem~\ref{thm:genupper-new} is exact?\end{question} 

 \medskip
 \noindent
 {\bf Conflicts of interest:}. The authors have no conflicts of interest.


\begin{thebibliography}{xx}

\bibitem{BaloghPete} J. Balogh, G. Pete, Random disease on the square grid, {\it Random Structures \& Algorithms} 13(3-4) (1998) 409--422. 

\bibitem{BBM} J. Balogh, B. Bollob\'{a}s, R. Morris, Bootstrap percolation in high dimensions, {\it Combinatorics, Probability and Computing} 19(5-6) (2010) 643--692.

\bibitem{BessyEtAl} S. Bessy, A. Bonato, J. Janssen, D. Rautenbach, E. Roshanbin, Burning a graph is hard, {\it Discrete Applied Mathematics} 232 (2017) 73--87.

\bibitem{Bsurvey} A. Bonato, A survey of graph burning, {\it Contributions to Discrete Mathematics} 16 (2021) 185-197.

\bibitem{BJR2014} A. Bonato, J. Janssen, E. Roshanbin, Burning a graph as a model of social contagion. In: Bonato, A., Graham, F., Pralat, P. (eds) Algorithms and Models for the Web Graph. WAW 2014. Lecture Notes in Computer Science, 8882. Springer, Cham. (2014) 13–22.

\bibitem{BJR2016} A. Bonato, J. Janssen, E. Roshanbin, How to burn a graph, {\it Internet Mathematics} 1-2 (2016) 85-100.

\bibitem{FirefighterSurvey} S. Finbow, G. MacGillivray, The firefighter problem: a survey of results, directions and questions, {\it Australasian Journal of Combinatorics} 43 (2009) 57-77. 

\bibitem{LQL2021} Y. Li, X. Qin, W. Li, The generalized burning number of graphs, {\it Applied Mathematics and Computation} 411 (2021) 126303.

\bibitem{MN} N. Morrison, J.A. Noel, Extremal bounds for bootstrap percolation in the hypercube, {\it Journal of Combinatorial Theory, Series A} 156 (2018) 61--84.

\bibitem{Pete1997} G. Pete, How to make the cube weedy? (in Hungarian), {\it Polygon} VII:1 (1997) 69--80.

\bibitem{PS} M. Przykucki, T. Shelton, Smallest percolating sets in bootstrap percolation on grids, {\it The Electronic Journal of Combinatorics} 27(4) 2020, \#P4.34.

\end{thebibliography}
\end{document}